\documentclass{amsproc}
\usepackage{eurosym}
\usepackage{amssymb}
\usepackage{amsfonts}
\def\Pp{\mathcal{P}}
\def\Qp{\mathcal{Q}}
\def\Qq{\mathcal{Q}}
\def\Cp{\mathcal{C}}

\theoremstyle{plain}
\newtheorem{theorem}{Theorem}

\newtheorem{question}[theorem]{Question}

\newtheorem{claim}[theorem]{Claim}

\newtheorem{corollary}[theorem]{Corollary}
\newtheorem{definition}[theorem]{Definition}

\newtheorem{lemma}[theorem]{Lemma}
\newtheorem{proposition}[theorem]{Proposition}

\numberwithin{equation}{section}
\newcommand{\field}[1]{\mathbb{#1}}
\newcommand{\C}{\field{C}}

\newcommand{\N}{\field{N}}
\newcommand{\Q}{\field{Q}}

\newcommand{\Z}{\field{Z}}
\newcommand{\z}{\field{Z}}

\newcommand{\superscript}[1]{\ensuremath{^{\textrm{#1}}}}

\def\wu{\superscript{*}}
\def\wg{\superscript{\dag}}

\begin{document}

\title{Non-normal very ample polytopes -- constructions and examples}

\author[Micha\l\ Laso\'{n}]{Micha\l\ Laso\'{n}\wu\footnote{\wu michalason@gmail.com; \'{E}cole Polytechnique F\'{e}d\'{e}rale de Lausanne, Chair of Combinatorial Geometry, EPFL-SB-MATHGEOM/DCG, Station 8, CH-1015 Lausanne, Switzerland and Institute of Mathematics of the Polish Academy of Sciences, ul.\'{S}niadeckich 8, 00-956 Warszawa, Poland}}

\author[Mateusz Micha\l ek]{Mateusz Micha\l ek\wg\footnote{\wg wajcha2@poczta.onet.pl; Institute of Mathematics of the Polish Academy of Sciences, ul.\'{S}niadeckich 8, 00-956 Warszawa, Poland}}

\thanks{Research supported by Polish National Science Centre grant no. 2012/05/D/ST1/01063 and by Swiss National Science Foundation Grants 200020-144531 and 200021-137574.}

\begin{abstract}
We present a method of constructing non-normal very ample polytopes as a segmental fibration of unimodular graph polytopes. In many cases we explicitly compute their invariants -- Hilbert function, Ehrhart polynomial, gap vector. In particular, we answer several questions posed by Beck, Cox, Delgado, Gubeladze, Haase, Hibi, Higashitani and Maclagan in \cite[Question 3.5 (1),(2), Question 3.6]{CoHaHi12}, \cite[Conjecture 3.5(a),(b)]{BeDeGuMi13}, \cite[Open question 3 (a),(b) p. 2310, Question p. 2316]{HaHiMa07}.
\end{abstract}

\maketitle

%%%%%%%%%%%%%%%%%%%%%%%%%%%%%%%%%%%%%%%%%%%%%%%%%%%%%%%%%%%%%%%%%%%%%%%%%%%%%%%%%%%%%%%%%%%%%%%%%%%%%%%%%%%%%%%%%%%%%%%%%%%%%%%%%%%%%%%%%%%%%%%%%%%%%%%%%%%%%%%%%%%%%%%%%%
%%%%%%%%%%%%%%%%%%%%%%%%%%%%%%%%%%%%%%%%%%%%%%%%%%%%%%%%%%%%%%%%%%%%%%%%%%%%%%%%%%%%%%%%%%%%%%%%%%%%%%%%%%%%%%%%%%%%%%%%%%%%%%%%%%%%%%%%%%%%%%%%%%%%%%%%%%%%%%%%%%%%%%%%%%
\section{Introduction}
%%%%%%%%%%%%%%%%%%%%%%%%%%%%%%%%%%%%%%%%%%%%%%%%%%%%%%%%%%%%%%%%%%%%%%%%%%%%%%%%%%%%%%%%%%%%%%%%%%%%%%%%%%%%%%%%%%%%%%%%%%%%%%%%%%%%%%%%%%%%%%%%%%%%%%%%%%%%%%%%%%%%%%%%%%
%%%%%%%%%%%%%%%%%%%%%%%%%%%%%%%%%%%%%%%%%%%%%%%%%%%%%%%%%%%%%%%%%%%%%%%%%%%%%%%%%%%%%%%%%%%%%%%%%%%%%%%%%%%%%%%%%%%%%%%%%%%%%%%%%%%%%%%%%%%%%%%%%%%%%%%%%%%%%%%%%%%%%%%%%%

The main object of our study are convex lattice polytopes. These combinatorial objects appear in many contexts including: toric geometry, algebraic combinatorics, integer programming, enumerative geometry and many others \cite{BrGu04, CoLiSc11, Fu93, KaMi14, LaMi14, St96}. Thus, it is not surprising that there is a whole hierarchy of their properties and invariants. Relations among them are of great interest.

One of the most intriguing and well-studied property of a lattice polytope is normality, or a related integral decomposition property. A polytope $\Pp$ is \emph{normal} in a lattice $M$ if for any $k\in\N$ every lattice point in $k\Pp$ is a sum of $k$ lattice points from $\Pp$.

Other, crucial property of a polytope, is very ampleness. A polytope $\Pp$ is \emph{very ample} in a lattice $M$ if for any sufficiently large $k\in\Z$ every lattice point in $k\Pp$ is a sum of $k$ lattice points from $\Pp$. This is equivalent to the fact  that for any vertex $v\in \Pp$ the monoid of lattice points in the real cone generated by $\Pp-v$ is generated by lattice points of $\Pp-v$. Obviously a normal polytope is very ample. The first example of a non-normal, very ample polytope was presented in \cite{BrGu02}. It is $5$-dimensional and corresponds to a triangulation of a real projective space.

It is worth to mention that a normal polytope defines a projectively normal toric embedding of the corresponding projective toric variety. Moreover, every projectively normal, equivariantly embedded toric variety is obtained in this way. However, not every normal projective toric variety has to be projectively normal. These defining normal projective toric varieties are exactly very ample polytopes.

Let us denote by $\Cp\subset\Z\times M$ the semigroup of lattice points in the real cone over $\{1\}\times\Pp$. Let $\Cp_j$ be the number of points $v\in\Cp$ with the zero coordinate $v_0=j$. The function $Ehr_\Pp:j\rightarrow|\Cp_j|$  is known as the \emph{Ehrhart polynomial}, and indeed it is a polynomial \cite{Eh62}. The function assigning to every $j$ the number of points $v$ in the semigroup generated by $\{1\}\times\Pp$ with $v_0=j$ is known as the \emph{Hilbert function} $H_\Pp$. For $j$ large enough it coincides with a polynomial, known as the \emph{Hilbert polynomial}. Clearly, the polytope $\Pp$ is normal if and only if the Ehrhart polynomial equals to the Hilbert function, that is if the cone $\Cp$ is generated by $\{1\}\times\Pp$. For a very ample polytope $\Pp$ the difference between the cone $\Cp$ and the semigroup generated by $\{1\}\times\Pp$ is a finite set (cf. \cite{BeDeGuMi13}), often referred to as the set of holes \cite{Hi12}. Vector which enumerates the number of holes $\gamma(\Pp)_j:=Ehr_\Pp(j)-H_\Pp(j)$ is called the gap vector.

Some properties of non-normal, very ample polytopes were already studied, see \cite{BeDeGuMi13,BrGu02, BrGu04,Hi12} and references therein. Moreover, many approaches to find new families of examples were presented in \cite{Br13}.

In this article we study relations among the above invariants (Ehrhart polynomial, Hilbert function, gap vector, and others) for non-normal, very ample polytopes. Strictly speaking, in Section \ref{Section2} we provide a new construction of very ample polytopes which are often non-normal. Our technique is based on lattice segmental fibrations (cf. \cite{BeDeGuMi13}) of unimodular polytopes. Recall that a polytope is called \emph{unimodular} if all its triangulations are unimodular, that is each simplex has the normalized volume equal to $1$. In Theorem \ref{TheoremVeryAmple} we prove that such a construction always yields a very ample polytope. In Section \ref{Section3}, in order to get examples with interesting properties, we specialize to a simple, natural class of unimodular polytopes -- edge polytopes corresponding to even cycles of length $2k$ and to the clique on $4$ vertices. For their segmental fibrations $\Pp_{k,a}$ (where parameter $a$ specifies the fibration of edge polytope of $C_{2k}$) and $\Qp_{a,b}$ (where parameters $a,b$ specify the fibration of edge polytope of $K_4$) we compute explicitly the Hilbert basis and the gap vector. Using these examples, in the last Section \ref{Section4}, we answer the following questions and conjectures.

\begin{question}\label{QuestionMain}
$ $
\begin{enumerate}
\item \cite[Question 3.5 (2)]{CoHaHi12} Is it true that the second dilatation $2\Pp$ of a very ample polytope $\Pp$ is always normal? \newline
Equivalently, is the second Veronese reembeding of an equivariantly embedded normal projective toric variety projectively normal?
\item \cite[Conjecture 3.5(a)]{BeDeGuMi13} The gap vector of a very ample polytope does not contain any internal zeros. \newline
Equivalently, suppose that $L$ is an ample line bundle on a normal toric variety $X$. If $H^0(X,L)^{\otimes n}$ surjects onto $H^0(X,L^{\otimes n})$ for some $n>1$, then the same is true for any $m\geq n$.
\item \cite[Conjecture 3.5(b)]{BeDeGuMi13} The gap vector of a very ample polytope $\Pp$ with normal facets in unimodal (that is $\gamma(\Pp)_1\leq\dots\leq\gamma(\Pp)_i\geq\gamma(\Pp)_{i-1}\geq\dots$ for some $i$).
\item \cite[Open problem 3 (a) p. 2310, Question p. 2316]{HaHiMa07}, \cite[Question 3.6]{CoHaHi12} Is it true that if $n\Pp$ and $m\Pp$ are normal, then so is $(n+m)\Pp$?  \newline
Equivalently, suppose that $L$ is an ample line bundle on a normal toric variety. Is it true that if $L^{\otimes n}$ and $L^{\otimes m}$ define projectively normal embeddings, then so deos $L^{\otimes (n+m)}$?
\item \cite[Open problem 3 (b) p. 2310]{HaHiMa07} Suppose that $\Pp$ and $\Qp$ are normal polytopes and the normal fan of $\Qp$ refines the normal fan of $\Pp$. Is $\Pp+\Qp$ normal? \newline
Equivalently, let $R$ be the region in the ample cone of a projective normal toric variety consisting of projectively normal line bundles. Is $R$ a module over the nef cone?
\item \cite[Question 3.5 (1)]{CoHaHi12} Does there exist a polytope $\Pp$ such that $\mu_{midp}(\Pp)<\mu_{Hilb}(\Pp)<\mu_{idp}(\Pp)$? (see Section \ref{Section2} for relevant definitions)
\end{enumerate}
\end{question}

During our research we used a lot  computer algebra systems \cite{BrLcRoSo,DeGrPfSc12,GaJo00}.

%%%%%%%%%%%%%%%%%%%%%%%%%%%%%%%%%%%%%%%%%%%%%%%%%%%%%%%%%%%%%%%%%%%%%%%%%%%%%%%%%%%%%%%%%%%%%%%%%%%%%%%%%%%%%%%%%%%%%%%%%%%%%%%%%%%%%%%%%%%%%%%%%%%%%%%%%%%%%%%%%%%%%%%%%%
\section*{Acknowledgements}
%%%%%%%%%%%%%%%%%%%%%%%%%%%%%%%%%%%%%%%%%%%%%%%%%%%%%%%%%%%%%%%%%%%%%%%%%%%%%%%%%%%%%%%%%%%%%%%%%%%%%%%%%%%%%%%%%%%%%%%%%%%%%%%%%%%%%%%%%%%%%%%%%%%%%%%%%%%%%%%%%%%%%%%%%%

We would like to thank Professor Winfried Bruns for help with computations.

%%%%%%%%%%%%%%%%%%%%%%%%%%%%%%%%%%%%%%%%%%%%%%%%%%%%%%%%%%%%%%%%%%%%%%%%%%%%%%%%%%%%%%%%%%%%%%%%%%%%%%%%%%%%%%%%%%%%%%%%%%%%%%%%%%%%%%%%%%%%%%%%%%%%%%%%%%%%%%%%%%%%%%%%%%
%%%%%%%%%%%%%%%%%%%%%%%%%%%%%%%%%%%%%%%%%%%%%%%%%%%%%%%%%%%%%%%%%%%%%%%%%%%%%%%%%%%%%%%%%%%%%%%%%%%%%%%%%%%%%%%%%%%%%%%%%%%%%%%%%%%%%%%%%%%%%%%%%%%%%%%%%%%%%%%%%%%%%%%%%%
\section{Constructions}\label{Section2}
%%%%%%%%%%%%%%%%%%%%%%%%%%%%%%%%%%%%%%%%%%%%%%%%%%%%%%%%%%%%%%%%%%%%%%%%%%%%%%%%%%%%%%%%%%%%%%%%%%%%%%%%%%%%%%%%%%%%%%%%%%%%%%%%%%%%%%%%%%%%%%%%%%%%%%%%%%%%%%%%%%%%%%%%%%
%%%%%%%%%%%%%%%%%%%%%%%%%%%%%%%%%%%%%%%%%%%%%%%%%%%%%%%%%%%%%%%%%%%%%%%%%%%%%%%%%%%%%%%%%%%%%%%%%%%%%%%%%%%%%%%%%%%%%%%%%%%%%%%%%%%%%%%%%%%%%%%%%%%%%%%%%%%%%%%%%%%%%%%%%%

%%%%%%%%%%%%%%%%%%%%%%%%%%%%%%%%%%%%%%%%%%%%%%%%%%%%%%%%%%%%%%%%%%%%%%%%%%%%%%%%%%%%%%%%%%%%%%%%%%%%%%%%%%%%%%%%%%%%%%%%%%%%%%%%%%%%%%%%%%%%%%%%%%%%%%%%%%%%%%%%%%%%%%%%%%
\subsection{Segmental fibrations of unimodular polytopes}
%%%%%%%%%%%%%%%%%%%%%%%%%%%%%%%%%%%%%%%%%%%%%%%%%%%%%%%%%%%%%%%%%%%%%%%%%%%%%%%%%%%%%%%%%%%%%%%%%%%%%%%%%%%%%%%%%%%%%%%%%%%%%%%%%%%%%%%%%%%%%%%%%%%%%%%%%%%%%%%%%%%%%%%%%%

Let us consider a slight modification of the definition \cite[Definition 2.2]{BeDeGuMi13} of a lattice segmental fibration.

\begin{definition}\label{DefinitionFibration}
A projection $f:\Q\times\Q^d\rightarrow\Q^d$ restricted to a lattice polytope $\Pp\subset\Q\times\Q^d$ is a \emph{lattice segmental fibration} if the preimage $f^{-1}(x)$ of every point $x\in f(\Pp)\cap\Z^{d}$ is a lattice segment of positive length.
\end{definition}
%Chimney, Nakajima

\begin{theorem}\label{TheoremVeryAmple}
If a polytope $\Pp$ admits a lattice segmental fibration $f$ to a unimodular polytope $\Qp:=f(\Pp)$, then it is very ample.
\end{theorem}

\begin{proof}
Let $\pi:\Q\times\Q^d\rightarrow \Q$ be the projection to the first factor, so that $\pi\times f$ is an identity. Let us define two functions:
$$h_l:\Qp\ni q\rightarrow\min_{x\in\Pp}\{\pi(x):f(x)=q\}\in\Q,\text{ }h_u:\Qp\ni q\rightarrow\max_{x\in\Pp}\{\pi(x):f(x)=q\}\in\Q.$$
Fix a vertex $v$ of the polytope $\Pp$. We may assume that $h_l(v)=\pi(v)$, as the case $h_u(v)=\pi(v)$ is analogous. To simplify the notation, we assume $v=0\in \z^{d+1}$. Let $C$ be the real cone pointed at $v$ and spanned by edges of $\Pp$ adjacent to $v$.
The domains of linearity of $h_l$ provide a partition of $\Qq$ into convex lattice polytopes. As $\Qq$ is unimodular we may extend this partition to a unimodular triangulation $T$. Consider the set $T'$ of those simplices in $T$ that contain $f(v)$. By forgetting those facets of simplices in $T'$ that do not contain $f(v)$, we may regard $T'$ as a unimodular subdivision of the projection by $f$ of the cone $C$. For each $t\in T'$ consider the cone $C_t$ spanned by the vectors $(h_l(l),l)\in\Z^{d+1}$ for $l\in t\cap\Z^d$, $l\neq v$ and the vector $(1,0,\dots,0)$. The cones $C_t$ are smooth, as each $t\in T'$ is unimodular and form a subdivision of $C$. In particular, $C$ as a semigroup is generated by the ray generators of the cones $C_t$ and these belong to $\Pp$.
\end{proof}

%%%%%%%%%%%%%%%%%%%%%%%%%%%%%%%%%%%%%%%%%%%%%%%%%%%%%%%%%%%%%%%%%%%%%%%%%%%%%%%%%%%%%%%%%%%%%%%%%%%%%%%%%%%%%%%%%%%%%%%%%%%%%%%%%%%%%%%%%%%%%%%%%%%%%%%%%%%%%%%%%%%%%%%%%%
\subsection{Products of polytopes and their invariants}
%%%%%%%%%%%%%%%%%%%%%%%%%%%%%%%%%%%%%%%%%%%%%%%%%%%%%%%%%%%%%%%%%%%%%%%%%%%%%%%%%%%%%%%%%%%%%%%%%%%%%%%%%%%%%%%%%%%%%%%%%%%%%%%%%%%%%%%%%%%%%%%%%%%%%%%%%%%%%%%%%%%%%%%%%%

An important invariant of a convex lattice polytope $\Pp$ is its Hilbert basis. It is the minimum set of generators, as a semigroup, of the cone $\Cp$. The Hilbert basis is always finite. A polytope is normal if and only if the Hilbert basis equals to $\{1\}\times\Pp$. Many invariants of lattice polytopes are connected with Hilbert basis. Let us introduce three of them, defined originally in \cite{BeDeGuMi13}, which we will need later.

\begin{definition}
For a convex lattice polytope $\Pp$:
\begin{enumerate}
\item $\mu_{Hilb}(\Pp)$ is the highest degree (zeroth coordinate of a point) an element of the Hilbert basis has,
\item $\mu_{midp}(\Pp)$ is the smallest positive integer, such that $\mu_{midp}\Pp$ is normal,
\item $\mu_{idp}(\Pp)$ is the smallest integer, such that for any $n\geq\mu_{idp}$ the polytope $n\Pp$ is normal.
\end{enumerate}
\end{definition}

\begin{lemma}\label{LemmaProduct}
For convex lattice polytopes $\Pp$ and $\Qp$ we have:
\begin{enumerate}
\item $\max\{\mu_{Hilb}(\Pp),\mu_{Hilb}(\Qp)\}\leq\mu_{Hilb}(\Pp\times \Qp)\leq f(\mu_{Hilb}(\Pp),\mu_{Hilb}(\Qp))$, for some function $f$ (with $f(2,3)\leq 6$),
\item $\mu_{midp}(\Pp\times \Qq)=\min\{n:n\Pp\text{ and }n\Qq\text{ are normal}\}$,
\item $\mu_{idp}(\Pp\times \Qq)=\max\{\mu_{idp}(\Pp),\mu_{idp}(\Qq)\}$.
\end{enumerate}
\end{lemma}

\begin{proof}
The last two statements follow from the fact that the product of polytopes is normal if and only if each polytope is. For the first statement, first note that if $v$ is an element in the Hilbert basis of $\Pp$ of degree $k$, then $v\times k\Qq$ is contained in the Hilbert basis of $\Pp\times \Qq$. The function $f$ can be defined in the following way. Consider a subset $A\subset \{1,\dots,\mu_{Hilb}(\Pp)\}$ of degrees appearing in the Hilbert basis of $\Pp$, and $B\subset \{1,\dots,\mu_{Hilb}(\Qq)\}$ for $\Qp$ analogously. Then $\mu_{Hilb}(\Pp\times \Qq)$ is at most the maximum degree of the Graver basis (that is the set of primitive binomials) of the toric ideal $I\subset\C[x_i,y_j:i\in A,j\in B]$ corresponding to the multiset of points $A\cup B\subset\Z$, cf. \cite[p. 33]{St96}. Here, we have to consider the grading in which a variable corresponding to $j\in\Z$ has degree $j$. Indeed, each element $(d,v,w)$ in the cone over $\{1\}\times\Pp\times\Qp$ yields two elements $v,w$ in the cones over $\Pp$ and $\Qp$. The decomposition of these elements into Hilbert basis elements gives a binomial $m_1-m_2$ in the toric ideal above. Any primitive binomial $n_1-n_2$ such that $n_i|m_i$ gives a decomposition of $(d,v,w)$. From this we can check that $f(2,3)\leq 6$.

Notice that it is enough to consider those elements of the Graver basis $n_1-n_2$, where $n_1$ is a monomial only in variables corresponding to the set $A$ and $n_2$ is a monomial only in variables corresponding to the set $B$.
\end{proof}

%%%%%%%%%%%%%%%%%%%%%%%%%%%%%%%%%%%%%%%%%%%%%%%%%%%%%%%%%%%%%%%%%%%%%%%%%%%%%%%%%%%%%%%%%%%%%%%%%%%%%%%%%%%%%%%%%%%%%%%%%%%%%%%%%%%%%%%%%%%%%%%%%%%%%%%%%%%%%%%%%%%%%%%%%%
%%%%%%%%%%%%%%%%%%%%%%%%%%%%%%%%%%%%%%%%%%%%%%%%%%%%%%%%%%%%%%%%%%%%%%%%%%%%%%%%%%%%%%%%%%%%%%%%%%%%%%%%%%%%%%%%%%%%%%%%%%%%%%%%%%%%%%%%%%%%%%%%%%%%%%%%%%%%%%%%%%%%%%%%%%
\section{Examples}\label{Section3}
%%%%%%%%%%%%%%%%%%%%%%%%%%%%%%%%%%%%%%%%%%%%%%%%%%%%%%%%%%%%%%%%%%%%%%%%%%%%%%%%%%%%%%%%%%%%%%%%%%%%%%%%%%%%%%%%%%%%%%%%%%%%%%%%%%%%%%%%%%%%%%%%%%%%%%%%%%%%%%%%%%%%%%%%%%
%%%%%%%%%%%%%%%%%%%%%%%%%%%%%%%%%%%%%%%%%%%%%%%%%%%%%%%%%%%%%%%%%%%%%%%%%%%%%%%%%%%%%%%%%%%%%%%%%%%%%%%%%%%%%%%%%%%%%%%%%%%%%%%%%%%%%%%%%%%%%%%%%%%%%%%%%%%%%%%%%%%%%%%%%%

Let $G$ be a graph with vertex set $V$ and edge set $E$. \emph{Edge polytope} $\Pp(G)$ of the graph $G$ is a polytope in the lattice $\Z^V$ with vertices $V(e)$ corresponding to edges $e\in E$. Points $V(e)\in\Z^V$ are defined by:
$$V(e)_v=
\begin{cases}
0 &\text{ if } v\notin e,\\
1 &\text{ if } v\in e.\\
\end{cases}$$
Polytopes $\Pp(G)$ were defined by Ohsugi and Hibi \cite{OhHi98}, see also \cite{OhHi99}.

\begin{proposition}[Herzog, Hibi, Ohsugi, \cite{OhHeHi00} Example 3.6 b)]\label{PropositionUnimodular}
Let $G$ be a connected graph. Edge polytope $\Pp(G)$ is unimodular if and only if $G$ does not contain two disjoint odd cycles.
\end{proposition}

In this Section we calculate Hilbert basis and gap vector of families of segmental fibrations of edge polytopes $\Pp_k:=\Pp(C_{2k})$ of an even cycle $C_{2k}$, and $\Qp:=\Pp(K_4)$ of the clique $K_4$. By the above proposition these polytopes are very ample, however it follows also from the fact that gap vectors are finite (see Theorems \ref{thm:gap}, \ref{thm:gap2}).

The reason why we consider graphs $C_{2k}$ and $K_4$ is that these are graphs with the property that every even cycle passes though all vertices, additionally they do not contain two disjoint odd cycles.

%%%%%%%%%%%%%%%%%%%%%%%%%%%%%%%%%%%%%%%%%%%%%%%%%%%%%%%%%%%%%%%%%%%%%%%%%%%%%%%%%%%%%%%%%%%%%%%%%%%%%%%%%%%%%%%%%%%%%%%%%%%%%%%%%%%%%%%%%%%%%%%%%%%%%%%%%%%%%%%%%%%%%%%%%%
\subsection{Definition of $\Pp_{k,a}$}
%%%%%%%%%%%%%%%%%%%%%%%%%%%%%%%%%%%%%%%%%%%%%%%%%%%%%%%%%%%%%%%%%%%%%%%%%%%%%%%%%%%%%%%%%%%%%%%%%%%%%%%%%%%%%%%%%%%%%%%%%%%%%%%%%%%%%%%%%%%%%%%%%%%%%%%%%%%%%%%%%%%%%%%%%%

Let us denote vertices of $C_{2k}$ appearing along the cycle by $1,\dots,2k$. We denote edges by $(i,i+1)$ and the corresponding vertices of the polytope $\Pp_k$ by $V(i,i+1)$. We are going to consider a polytope $\Pp_{k,a}\subset\Z^k\times\Z$ defined by vertices
$$(V(i,i+1),0),(V(i,i+1),1)\text{ for }i=2,3,\dots,2k,$$
$$\text{and }(V(1,2),a),(V(1,2),a+1).$$
Clearly, projection $f:\Z^{2k}\times\Z\ni\Pp_{k,a}\rightarrow\Pp_k\in\Z^{2k}$ is a lattice segmental fibration. Let us denote the cone over $1\times\Pp_{k}$ by $\Cp_{k}$, and the cone over $1\times\Pp_{k,a}$ by $\Cp_{k,a}$ (with $1$ on the $0$-th coordinate). We can extend projection $f$ to
$$f:\Z\times\Z^{2k}\times\Z\supset \Cp_{k,a}\rightarrow\Cp_k\subset\Z\times\Z^{2k}.$$
Let $B_{k}$ be the set of vertices of $1\times\Pp_{k}$, and $B_{k,a}$ be the set of vertices of $1\times\Pp_{k,a}$. Let $\mathbf{1}:=(1,\dots,1)\in\Z^{2k}$, and let $A_{k,a}:=(k,\mathbf{1},[k+1,a-1])$ be the set of $k-a-1$ points in $\Z\times\Z^{2k}\times\Z$.

Observe that points $v$ in the cones $\Cp_k,\Cp_{k,a}$ satisfy the following:
\begin{equation}\label{1}
v_0=v_1+v_3+\dots+v_{2k-1}=v_2+v_4+\dots+v_{2k},
\end{equation}
\begin{equation}\label{2}
0\leq v_i\leq v_{i-1}+v_{i+1}\text{ for every }i=1,\dots,2k.
\end{equation}
Points in the cone $\Cp_{k,a}$ additionally satisfy two more inequalities:
\begin{equation}\label{3}
0\leq v_{2k+1}\leq (a+1)v_1+v_3+\dots+v_{2k-1},
\end{equation}
\begin{equation}\label{4}
0\leq v_{2k+1}\leq (a+1)v_2+v_4+\dots+v_{2k}.
\end{equation}
It is also not hard to argue that indeed the above equalities and inequalities define these cones.

We will need two lemmas about the cone $\Cp_{k}$.

\begin{lemma}\label{Lemma0}
Suppose a point $v\in\Cp_{k}$ satisfies $v_i=0$ for some $i\in\{1,\dots,2k\}$. Then $v$ has a unique expression as a non-negative linear combination of points from $B_{k}$, moreover the coefficients are integers.
\end{lemma}
\begin{proof}
It is easy to prove by induction on $j$ that the coefficient of $V(j,j+1)$ is equal to $v_j-v_{j-1}+v_{j-2}-v_{j-3}+\dots\pm v_i.$
\end{proof}

\begin{lemma}\label{Lemma1}
Suppose there is an equality $V(e_1)+\dots+V(e_n)=V(e_1')+\dots+V(e_n')$ in the cone $\Cp_{k}$ for some $e_i,e_i'\in E(C_{2k})$. Then the formal difference of multisets $\{e_1,\dots,e_n\}-\{e_1',\dots,e_n'\}$ is a multiple of $\{(1,2),\dots,(2k-1,2k)\}-\{(2,3),\dots,(2k,1)\}$.
\end{lemma}
\begin{proof}
If an edge $(i-1,i)$ is with $+$ sign, then by looking at vertex $i$ we get that $(i,i+1)$ is with sign $-$. Therefore we can find $\{(1,2),\dots,(2k-1,2k)\}-\{(2,3),\dots,(2k,1)\}$. We can subtract it, and the rest follows by induction.
\end{proof}

%%%%%%%%%%%%%%%%%%%%%%%%%%%%%%%%%%%%%%%%%%%%%%%%%%%%%%%%%%%%%%%%%%%%%%%%%%%%%%%%%%%%%%%%%%%%%%%%%%%%%%%%%%%%%%%%%%%%%%%%%%%%%%%%%%%%%%%%%%%%%%%%%%%%%%%%%%%%%%%%%%%%%%%%%%
\subsection{Hilbert basis of $\Pp_{k,a}$}
%%%%%%%%%%%%%%%%%%%%%%%%%%%%%%%%%%%%%%%%%%%%%%%%%%%%%%%%%%%%%%%%%%%%%%%%%%%%%%%%%%%%%%%%%%%%%%%%%%%%%%%%%%%%%%%%%%%%%%%%%%%%%%%%%%%%%%%%%%%%%%%%%%%%%%%%%%%%%%%%%%%%%%%%%%

\begin{theorem}
The set $B_{k,a}\cup A_{k,a}$ is the Hilbert basis of $\Cp_{k,a}$.
\end{theorem}

\begin{proof}
We will apply the following lemma to the last, $(2k+1)$-st, coordinate.

\begin{lemma}\label{LemmaInteger}
Suppose $c=z_1+\dots+z_r$ for some $z_i\in [a_i,b_i]$, where $a_i,b_i\in\Z$. If  $c$ is an integer, then $c=z'_1+\dots+z'_r$ for some $z'_i\in [a_i,b_i]\cap\Z$.
\end{lemma}
\begin{proof}
Suppose $c=z_1+\dots+z_r$ for $z_i\in [a_i,b_i]$ and the number of non-integers among $z_i$ is minimum. If there is a non-integer, then since $c$ is an integer, there must be at least two non-integers. In this case one can increase one of them and decrease the other until one of them (at least) reaches an integer. They still belong to the corresponding intervals, contradicting the minimality of non-integers among $z_i$.
\end{proof}

Clearly, all vertices of $1\times\Pp_{k,a}$, that is the set $B_{k,a}$, must be in the Hilbert basis of $\Cp_{k,a}$.

\begin{claim}\label{Claim1}
Elements of $B_{k,a}$ generate all lattice points $v\in\Cp_{k,a}$ satisfying $v_i=0$ for some $i\in\{1,\dots,2k\}$.
\end{claim}
\begin{proof}
Since $v\in\Cp_{k,a}$ we have that $v$ is a non-negative linear combination of elements of $B_{k,a}$. Now $f(v)\in\Cp_{k}$ satisfies $f(v)_i=0$. Due to Lemma \ref{Lemma0} the sums of coefficients of vertices corresponding to each edge are non-negative integers. Using Lemma \ref{LemmaInteger} one can adjust the last coordinate, that is assure that all coefficients are non-negative integers.
\end{proof}

Claim \ref{Claim1} in particular means that elements of $B_{k,a}$ generate all points $v\in\Cp_{k,a}$ with $v_0<k$. This is because (\ref{1}) $v_1+\dots+v_{2k}=2v_0<2k$ and $v_i$ are non-negative integers, so we get that $v_i=0$ for some $i\in\{1,\dots,2k\}$.

Let us consider points $v\in\Cp_{k,a}$ with $v_0=k$. If $v_i=0$ for some $i\in\{1,\dots,2k\}$, then by Claim \ref{Claim1} $v$ is generated by elements of $B_{k,a}$. Otherwise, for every $i\in\{1,\dots,2k\}$ we have $v_i=1$. From (\ref{3}) it follows that $v_{2k+1}\in\{0,\dots,a+k\}$.

Element $v\in\Cp_{k,a}$ with $v_0=k$ and $v_1=\dots=v_{2k}=1$ can be achieved only in two ways as a non-negative integer linear combination of elements $B_{k,a}$. By taking vertices corresponding to edges $(1,2),(3,4),\dots,(2k-1,2k)$, for each edge one vertex. Then $v_{2k+1}\in\{0,\dots,k\}$. Or, by taking vertices corresponding to edges $(2,3),(4,5),\dots,(2k,1)$, for each edge one vertex. Then $v_{2k+1}\in\{a,\dots,a+k\}$. Elements of $A_{k,a}$ are exactly the missing ones.

Consider points $v\in\Cp_{k,a}$ with $v_0>k$. We will prove by induction on $v_k$ that they are generated by the set $B_{k,a}\cup A_{k,a}$. Since (\ref{1}) $v_1+\dots+v_{2k}=2v_0>2k$ we have that $v_i\geq 2$ for some $i\in\{1,\dots,2k\}$.  Point $v$ is a non-negative linear combination of elements of $B_{k,a}$. The sum of coefficients of vertices corresponding to the edge $(i-1,i)$ and to the edge $(i,i+1)$ equals to $v_i$. Without loss of generality we can assume that the sum of coefficients corresponding to the edge $(i,i+1)$ is greater or equal to $1$. Now, similarly to the proof of Lemma \ref{LemmaInteger}, we can assure that the coefficient of a vertex $w\in B_{k,a}$ corresponding to the edge $(i,i+1)$ is greater or equal to $1$. Then $v-w\in\Cp_{k,a}$, and the assertion follows by induction.
\end{proof}

%%%%%%%%%%%%%%%%%%%%%%%%%%%%%%%%%%%%%%%%%%%%%%%%%%%%%%%%%%%%%%%%%%%%%%%%%%%%%%%%%%%%%%%%%%%%%%%%%%%%%%%%%%%%%%%%%%%%%%%%%%%%%%%%%%%%%%%%%%%%%%%%%%%%%%%%%%%%%%%%%%%%%%%%%%
\subsection{Gap vector of $\Pp_{k,a}$}
%%%%%%%%%%%%%%%%%%%%%%%%%%%%%%%%%%%%%%%%%%%%%%%%%%%%%%%%%%%%%%%%%%%%%%%%%%%%%%%%%%%%%%%%%%%%%%%%%%%%%%%%%%%%%%%%%%%%%%%%%%%%%%%%%%%%%%%%%%%%%%%%%%%%%%%%%%%%%%%%%%%%%%%%%%

\begin{theorem}\label{thm:gap}
The gap vector of $\Pp_{k,a}$ equals to:
$$\gamma(\Pp_{k,a})_i=
\begin{cases}
0 &\text{ if }i<k,\text{ or }a-2<i,\\
(a-i-1){{i+k-1}\choose{2k-1}}&\text{ if } k\leq i\leq a-2.\\
\end{cases}$$
\end{theorem}

\begin{proof}
We will describe the set of gaps explicitly. Let $\mathbf{1}:=(1,\dots,1)\in \Z^{2k}$. Recall that for an edge $e\in C_{2k}$ we have $h_u(V(e))=1,h_l(V(e))=0$ if $e\neq (1,2)$, and $h_u(V(e))=a+1,h_l(V(e))=a$ if $e=(1,2)$. For a multiset $M$ of edges from $C_{2k}$ let $h_u(M),h_l(M),S(M)$ be the sums of functions $h_u(V(e)),h_l(V(e)),V(e)$ accordingly over elements of $e\in M$.

\begin{claim}\label{Claim2}
The set of gaps at level $i+k$ (having zero coordinate equal to $k+i$) is the union of the following disjoint sets:
$$(k+i,\mathbf{1}+S(M),[k+1+h_u(M),a-1+h_l(M)]),$$
over all multisubsets $M\subset E(C_{2k})$ of cardinality $i$.
\end{claim}
\begin{proof}
If $S(M_1)=S(M_2)$, then the formal difference $M_1-M_2$ is a multiple of $(1,2)+\dots+(2k-1,2k)-(2,3)-\dots-(2k,1)$. Then, functions $h_l,h_u$ differ by a multiple of $a$, so the sets are indeed disjoint.

%Consider an element $v\in\Cp_{k,a}$. Due to Claim \ref{Claim1} if $v_i=0$ for some $i\in\{1,\dots,2k\}$, then $v$ is generated by elements of $B_{k,a}$, so it is not a whole.

By Proposition \ref{PropositionUnimodular} polytope $\Pp_{k}$ is unimodular, hence normal. Thus points of $\Cp_{k,a}$ are of the form $(i,S(M),x)$ for some multiset of edges $M$ and integer $x$. Let us fix $i$ and $S:=S(M)$. Consider all multisets $M$ such that $S(M)=S$. Due to Lemma \ref{Lemma1} they are exactly:
$$M'+c\{(1,2),\dots,(2k-1,2k)\},\dots$$
$$\dots,M'+(c-r)\{(1,2),\dots,(2k-1,2k)\}+r\{(2,3),\dots,(2k,1)\},\dots$$
$$\dots,M'+c\{(2,3),\dots,(2k,1)\}\text{ for some $M'$}.$$
Elements of the lattice generated by $\{1\}\times\Pp_{k,a}$ corresponding to a multiset $M$ have the last coordinate $x(M)\in [h_l(M),h_u(M)]$. Sets
$$(k+i,\mathbf{1}+S(M),[k+1+h_u(M),a-1+h_l(M)])$$ are exactly the gaps between consecutive intervals.
\end{proof}

For every multiset of edges $M$ of cardinality $i$ we have
$$a-1+h_l(M)-(k+1+h_u(M))=a-i-k-1.$$
Moreover, the number of multisets of cardinality $i$ of a $2k$-element set equals to ${{i+2k-1}\choose{2k-1}}$. Multiplying these numbers we get the assertion.
\end{proof}

\begin{corollary}\label{Corollaryk}
The polytope $\Pp_{k,k+2}$ has exactly one gap which is in degree $k$, that is $\gamma(\Pp_{k,k+2})=(0,\dots,0,1)$, where $1$ is on position $k$. Moreover, the polytope $s\Pp_{k,k+2}$ is normal if and only if $s$ does not divide $k$ or $s\geq k$. In particular:
\begin{enumerate}
\item $\mu_{Hilb}(\Pp_{k,k+2})=k$,
\item $\mu_{midp}(\Pp_{k,k+2})$ equals to the smallest non divisor of $k$,
\item $\mu_{idp}(\Pp_{k,k+2})$ equals to the highest proper divisor of $k$ plus $1$.
\end{enumerate}
\end{corollary}

\begin{proof}
The first sentence follows directly by Theorem \ref{thm:gap}. In particular, there is exactly one element of the Hilbert basis $v$ of degree greater than $1$. If $s$ properly divides $k$, then $v\in \frac{k}{s}(s\Pp)$ shows that $s\Pp$ is non-normal. If $s$ does not divide $k$, then any lattice point in $m(s\Pp)$ is not a hole, so it is a sum of $ms$ integral points from $\Pp$. In particular, $s\Pp$ is normal.
\end{proof}

%%%%%%%%%%%%%%%%%%%%%%%%%%%%%%%%%%%%%%%%%%%%%%%%%%%%%%%%%%%%%%%%%%%%%%%%%%%%%%%%%%%%%%%%%%%%%%%%%%%%%%%%%%%%%%%%%%%%%%%%%%%%%%%%%%%%%%%%%%%%%%%%%%%%%%%%%%%%%%%%%%%%%%%%%%
\subsection{Definition of $\Qp_{a,b}$}
%%%%%%%%%%%%%%%%%%%%%%%%%%%%%%%%%%%%%%%%%%%%%%%%%%%%%%%%%%%%%%%%%%%%%%%%%%%%%%%%%%%%%%%%%%%%%%%%%%%%%%%%%%%%%%%%%%%%%%%%%%%%%%%%%%%%%%%%%%%%%%%%%%%%%%%%%%%%%%%%%%%%%%%%%%

We consider the clique $K_4$ on vertices $1,2,3,4$. We denote edges by $(i,j)$ and the corresponding vertices of the octahedron $\Qp:=\Pp(K_4)$ by $V(i,j)$. We are going to consider a polytope $\Qp_{a,b}\subset\Z^4\times\Z$ defined by vertices
$$(V(1,2),0),(V(1,2),1),$$
$$(V(2,3),0),(V(2,3),1),$$
$$(V(1,3),0),(V(1,3),1),$$
$$(V(4,1),0),(V(4,1),b),$$
$$(V(4,2),b+4a+2),(V(4,2),2b+4a+2),$$
$$(V(4,3),2b+11a+4),(V(4,3),3b+11a+4).$$
Clearly, projection $f:\Z^4\times\Z\ni\Qp_{a,b}\rightarrow\Qp\in\Z^4$ is a lattice segmental fibration.

\begin{lemma}\label{lem:normalfaces}
If a lattice polytope $\Pp$ is obtained by a lattice segmental fibration $f$ over $\Pp(K_4)$, then it has normal facets.
\end{lemma}

\begin{proof}
The facets $F$ of $\Pp$ are of two types. Either $\dim f(F)=\dim F-1$ or $\dim f(F)=\dim F$.

In the first case $f(F)$ is a facet of $\Pp(K_4)$, hence a unimodular simplex. Hence, $F$ is a lattice segmental fibration over a unimodular simplex. Any lattice segmental fibration over a unimodular simplex is a smooth, normal Nakajima polytope \cite[Theorem 4.2]{BeDeGuMi13}, \cite{Na86}, \cite[Section 2.2.1]{HaPaPiSa14}.

In the second case, as in the proof of Theorem \ref{TheoremVeryAmple}, we know that $f(F)$ has a unimodular triangulation, in particular is normal. We claim that the restriction of $f$ to lattice points in the affine space containing $F$ is a bijection onto the lattice generated by $\Pp(K_4)$. Indeed, it is an injection, as it preserves dimension. Moreover, each point of a unimodular simplex in $f(F)$ can be lifted, by the definition of the lattice segmental fibration, to a point of $F$, so the map is surjective. Hence, $F$ and $f(F)$ are isomorphic as lattice polytopes.
\end{proof}

%%%%%%%%%%%%%%%%%%%%%%%%%%%%%%%%%%%%%%%%%%%%%%%%%%%%%%%%%%%%%%%%%%%%%%%%%%%%%%%%%%%%%%%%%%%%%%%%%%%%%%%%%%%%%%%%%%%%%%%%%%%%%%%%%%%%%%%%%%%%%%%%%%%%%%%%%%%%%%%%%%%%%%%%%%
\subsection{Gap vector of $\Qp_{a,b}$}
%%%%%%%%%%%%%%%%%%%%%%%%%%%%%%%%%%%%%%%%%%%%%%%%%%%%%%%%%%%%%%%%%%%%%%%%%%%%%%%%%%%%%%%%%%%%%%%%%%%%%%%%%%%%%%%%%%%%%%%%%%%%%%%%%%%%%%%%%%%%%%%%%%%%%%%%%%%%%%%%%%%%%%%%%%

\begin{theorem}\label{thm:gap2}
Suppose that $b>7a$. Then, the gap vector of $\Qp_{a,b}$ equals to:
$$\gamma(\Qp_{a,b})_{i+2}=\binom{i+2}{2}(\max\{4a-i,0\}+\max\{7a-i,0\})\text{ for }i\geq 0.$$
\end{theorem}

\begin{proof}
We will describe the set of gaps explicitly. Let $\mathbf{1}:=(1,1,1,1)\in\Z^4$.

\begin{claim}\label{Claim4}
The set of gaps at level $i+2$ (having zero coordinate equal to $2+i$) is the union of the following disjoint sets:
$$(2+i,\mathbf{1}+S(M),[b+2+h_u(M),b+4a+1+h_l(M)]),$$
$$(2+i,\mathbf{1}+S(M),[2b+4a+4+h_u(M),2b+11a+3+h_l(M)]),$$
over all multisubsets $M\subset\{(1,2),(2,3),(1,3)\}$ of cardinality $i$.
\end{claim}

\begin{proof}
The proof goes similarly to the proof of Claim \ref{Claim2}. The only difference is that it is enough to consider multisets $M\subset\{(1,2),(2,3),(1,3)\}$. It is because the lengths of segments corresponding to edges $(4,1),(4,2),(4,3)$ are equal to $b$, so for them $h_u-h_l$ is greater than the lengths of intervals of gaps at level two, that is $4a,7a$. Thus the corresponding set of gaps would be empty.
\end{proof}

For every multiset of edges $M$ of cardinality $i$ the sum of lengths of intervals
$$(2+i,\mathbf{1}+S(M),[b+2+h_u(M),b+4a+1+h_l(M)]),$$
$$(2+i,\mathbf{1}+S(M),[2b+4a+4+h_u(M),2b+11a+3+h_l(M)])$$
equals to $\max\{4a-i,0\}+\max\{7a-i,0\}$.
The number of multisets of cardinality $i$ of a $3$-element set equals to $\binom{i+2}{2}$. Multiplying these numbers we get the assertion.
\end{proof}

%%%%%%%%%%%%%%%%%%%%%%%%%%%%%%%%%%%%%%%%%%%%%%%%%%%%%%%%%%%%%%%%%%%%%%%%%%%%%%%%%%%%%%%%%%%%%%%%%%%%%%%%%%%%%%%%%%%%%%%%%%%%%%%%%%%%%%%%%%%%%%%%%%%%%%%%%%%%%%%%%%%%%%%%%%
%%%%%%%%%%%%%%%%%%%%%%%%%%%%%%%%%%%%%%%%%%%%%%%%%%%%%%%%%%%%%%%%%%%%%%%%%%%%%%%%%%%%%%%%%%%%%%%%%%%%%%%%%%%%%%%%%%%%%%%%%%%%%%%%%%%%%%%%%%%%%%%%%%%%%%%%%%%%%%%%%%%%%%%%%%
\section{Applications}\label{Section4}
%%%%%%%%%%%%%%%%%%%%%%%%%%%%%%%%%%%%%%%%%%%%%%%%%%%%%%%%%%%%%%%%%%%%%%%%%%%%%%%%%%%%%%%%%%%%%%%%%%%%%%%%%%%%%%%%%%%%%%%%%%%%%%%%%%%%%%%%%%%%%%%%%%%%%%%%%%%%%%%%%%%%%%%%%%
%%%%%%%%%%%%%%%%%%%%%%%%%%%%%%%%%%%%%%%%%%%%%%%%%%%%%%%%%%%%%%%%%%%%%%%%%%%%%%%%%%%%%%%%%%%%%%%%%%%%%%%%%%%%%%%%%%%%%%%%%%%%%%%%%%%%%%%%%%%%%%%%%%%%%%%%%%%%%%%%%%%%%%%%%%

Observe first that due to Proposition \ref{PropositionUnimodular} edge polytopes $\Pp(C_{2k})$ are unimodular. Hence, due to Theorem \ref{TheoremVeryAmple} their segmental fibrations $\Pp_{k,a}$ are very ample polytopes. It also follows from Theorem \ref{thm:gap}, since the number of gaps is finite. We are going to answer Question \ref{QuestionMain} $(1)-(6)$:

\begin{enumerate}
\item No, due to Corollary \ref{Corollaryk} all dilatations up to $n$ of the polytope $\Pp_{n!,n!+2}$ are non-normal.

\item False. The polytope $\Pp_{k,k+2}$ (for any $k>2$) is a counterexample (it has even normal facets), since its gap vector has all internal entries equal to zero.

By considering the product $\Pp_{k_1,k_1+2}\times\Pp_{k_2,k_2+2}$ for $k_1\neq k_2$ we obtain a very ample polytope with the gap vector having exactly two nonzero entries at positions $k_1$ and $k_2$. This is a refined counterexample to \cite[Conjecture 3.5(a)]{BeDeGuMi13}. However, the facets of this polytope are non-normal, thus it is not a counterexample to \cite[Conjecture 3.5(b)]{BeDeGuMi13}.

\item False. Take $a,b\geq 1$ such that $b>7a$, and consider polytope $\Qp_{a,b}$. By Lemma \ref{lem:normalfaces} this polytope has normal facets. From Theorem \ref{thm:gap2} follows that $\gamma_{4a+1}(\Qp_{a,b})>\gamma_{4a+2}(\Qp_{a,b})$. It is because:
$$\binom{4a+1}{2}(3a+2)>\binom{4a+2}{2}(3a),$$
$$(4a)(3a+2)>(4a+2)(3a).$$
Similarly $\gamma_{4a+3}(\Qp_{a,b})>\gamma_{4a+2}(\Qp_{a,b})$. It is because:
$$\binom{4a+3}{2}(3a-1)>\binom{4a+2}{2}(3a),$$
$$(4a+3)(3a-1)>(4a+1)(3a).$$
Thus the gap vector of $\Qp_{a,b}$ is not unimodal.

\item No. By adding two non divisors $n_1,n_2$ of $k$ to a proper divisor of $k$, we obtain a polytope $(n_1+n_2)\Pp_{k,k+2}$ which is non-normal, while $n_1\Pp_{k,k+2}$ and $n_2\Pp_{k,k+2}$ are. In particular, $2\Pp_{25,27}$ and $3\Pp_{25,27}$ are normal, while $5\Pp_{25,27}$ is not. Notice that the answer to the question is positive when $\dim\Pp\leq 6$.

\item No, the same example as above -- see reformulation in \cite[Open problem 3 (b) p. 2310]{HaHiMa07}.

\item Yes. Let us recall that in \cite[Example 2.3]{CoHaHi12} the authors constructed a family of polytopes $\Pp_l$ for which $\mu_{midp}(\Pp_l)=\mu_{Hilb}(\Pp_l)=2$ and $\mu_{idp}(\Pp_l)=2l$. Consider the product $\Qp:=\Pp_{3,5}\times \Pp_l$. By Lemma \ref{LemmaProduct} we have:
$$\mu_{midp}(\Qq)=2<3\leq \mu_{Hilb}(\Qq)\leq 6<2l=\mu_{idp}(\Qq),\text{ for }l>3.$$
\end{enumerate}

In \cite[Question 3.5 (1)]{CoHaHi12} the authors ask for the relations among $\mu_{midp}(\Pp),\mu_{idp}(\Pp)$ and $\mu_{Hilb}(\Pp)$. As our examples show this relations can be quite complicated.

\begin{corollary}\label{Corollary2}
$ $
\begin{enumerate}
\item There exists an integral polytope $\Pp$ with $\mu_{idp}(\Pp)=2$ and $\mu_{Hilb}(\Pp)=n$ if and only if $n$ is a prime number.
\item There exists an integral polytope $\Pp$ with $\mu_{idp}(\Pp)=3$ and $\mu_{Hilb}(\Pp)=n$ if and only if $n=4$ or $n$ is a prime number different from $2$.
\item Consider a prime number $n$ greater or equal to a positive integer $k$. There exists an integral polytope $\Pp$ with $\mu_{idp}(\Pp)=k$ and $\mu_{Hilb}(\Pp)=n$.
\end{enumerate}
\end{corollary}

\begin{proof}
Notice that $\mu_{Hilb}(\Pp)$ cannot have a proper divisor greater than $\mu_{idp}(\Pp)$. Thus, if $\mu_{Hilb}(\Pp)>(\mu_{idp}(\Pp)-1)^2$, then $\mu_{Hilb}(\Pp)$ must be a prime number (as showed below the inequality must be strict). Surprisingly, even if $\mu_{idp}(\Pp)=2$, still $\mu_{Hilb}(\Pp)$ may be an arbitrary prime integer $p$. Indeed, it is sufficient to consider $\Pp=\Pp_{p,p+2}$.

For the second statement we first prove the implication $\Rightarrow$. Indeed, $n\neq 2$, as if $n=2$, then $\mu_{midp}(\Pp)=2$ which contradicts $\mu_{idp}(\Pp)=3$. If $n>5$ then $n$ must be prime by the arguments presented above. To prove $\Leftarrow$ we need to present examples for each $n$. For $n=3$ the construction is given in \cite[Theorem 2.1]{CoHaHi12}. For $n=4$ the construction is given in \cite[Theorem 2.6]{CoHaHi12}.

It remains to prove the last statement. One can consider the product of two polytopes -- the example for \cite[Theorem 2.1]{CoHaHi12} for $j=k$ and $\Pp_{n,n+2}$.
\end{proof}

%%%%%%%%%%%%%%%%%%%%%%%%%%%%%%%%%%%%%%%%%%%%%%%%%%%%%%%%%%%%%%%%%%%%%%%%%%%%%%%%%%%%%%%%%%%%%%%%%%%%%%%%%%%%%%%%%%%%%%%%%%%%%%%%%%%%%%%%%%%%%%%%%%%%%%%%%%%%%%%%%%%%%%%%%%
%%%%%%%%%%%%%%%%%%%%%%%%%%%%%%%%%%%%%%%%%%%%%%%%%%%%%%%%%%%%%%%%%%%%%%%%%%%%%%%%%%%%%%%%%%%%%%%%%%%%%%%%%%%%%%%%%%%%%%%%%%%%%%%%%%%%%%%%%%%%%%%%%%%%%%%%%%%%%%%%%%%%%%%%%%

\end{document}